\documentclass{article}
\usepackage{amsmath, amsthm, amsfonts}

\usepackage{graphicx}

\theoremstyle{plain}
\newtheorem{theorem}{Theorem}
\newtheorem*{lemma}{Lemma}
\newtheorem*{liouville}{Liouville's theorem}
\newtheorem*{roth}{Roth's theorem}

\usepackage{hyperref}

\title{A Couple of Transcendental\\
Prime-Representing Constants}
\author{Juan Luis Varona}
\date{}

\begin{document}
%---------------

%-----------
\maketitle
%-----------

\makeatletter
\def\blfootnote{\gdef\@thefnmark{}\@footnotetext}
\makeatother
\blfootnote{\textbf{This paper has been published as:}
J. L. Varona, 
A Couple of Transcendental Prime-Representing Constants,
\textit{Amer.\ Math.\ Monthly}\ \textbf{128} (2021), no.~10, 922--928, 
\url{https://doi.org/10.1080/00029890.2021.1977885}}

%-----------
\begin{abstract}
It is well known that the arithmetic nature of Mills' prime-representing constant is uncertain: we do not know if Mills' constant is a rational or irrational number. In the case of other prime-representing constants, irrationality can be proved, but it is not known whether these constants are algebraic or transcendental numbers.
By using Liouville or Roth's theorems about approximation by rationals, we find a couple of prime-representing constants that can be proved to be transcendental numbers.
\end{abstract}
%-----------

%-----------
\section{Introduction}
\label{sec:intro}
%-----------

A prime-representing constant is a constant that, by means of some process, generates infinitely many prime numbers, sometimes all the prime numbers.

The first and best known result about prime-representing constants was established by Mills in 1947 \cite{Mi}. He proved that there exists a real number $\theta$ such that $\lfloor \theta^{3^n} \rfloor$ is a prime for every positive integer $n$ (and where the so-called floor function $\lfloor \cdot \rfloor$ is used to denote the integer part). 
% floor function
The proof of this fact depends on a result of Ingham that is not elementary \cite{Ing}: if $p_n$ denotes the $n$th prime, then there is a constant $K$ such that $p_{n+1}-p_n < K p_n^{5/8}$ for $n= 1,2,\dots$. Assuming Ingham's theorem, finding Mills' constant $\theta$ is not difficult. 

There is no closed-form formula known for such $\theta$ and, actually, there are uncountably many possible values of $\theta$ with the prime-representing property. Hence, if one wants to define Mills' constant, a good approach is to look for the least $\theta$ such that $\lfloor \theta^{3^n} \rfloor$ is prime for $n= 1,2,\dots$. In this way, and by making some unproved but reasonable assumptions (in particular, the Riemann hypothesis), Mills' constant can be computed with precision. 
For instance, \cite{CaCh} shows a method that gives $\theta = 1.3063778838\dots$ to over 6850 decimal places.
But it is not known whether $\theta$ is a rational number. 

The primes generated using Mills' procedure grow very fast. If we denote $M_n = \lfloor \theta^{3^n} \rfloor$, one has
\[
  M_1 = 2,\quad M_2 = 11,\quad M_3 = 1361,\quad M_4 = 2\,521\,008\,887,
\] 
and $M_5$ has $29$ digits (OEIS sequence A051254, see~\cite{OEIS-Mi}).

In 1951, Wright \cite{Wr} found another prime-representing constant using a method that, instead of relying on Ingham's theorem, depends on Bertrand's postulate, a much more elementary result: $p_{n+1}-p_n < p_n$ (or, as usually stated, there is always a prime between $N$ and $2N$). Using this, Wright proves the existence of a constant $\alpha \in (1, 2)$ such that, if we recursively define $\alpha_0 = \alpha$ and $\alpha_{n+1} = 2^{\alpha_n}$, then $\lfloor \alpha_n \rfloor$ is prime for $n= 1,2,\dots$ 
(said more compactly, 
%$\big\lfloor 2^{2^{\cdot^{\cdot^{\cdot^{2^\alpha}}}}} \big\rfloor$ 
%$\big\lfloor 2^{2^{\dots^{2^\alpha}}} \big\rfloor$ 
%$\big\lfloor 2^{2^{\rotatebox{15}{\footnotesize\dots}^{2^\alpha}}} \big\rfloor$ 
$\big\lfloor 2^{2^{\rotatebox{15}{\footnotesize\dots}^{\raisebox{-0.7pt}{$\scriptscriptstyle 2^\alpha$}}}} \!\big\rfloor$ 
is a prime for any number of iterations of the exponential).
Again, the method to obtain $\alpha$ is not unique, and, in this case, for historical reasons (it is the example proposed by Wright), it is customary to take the biggest $\alpha$ with that property; this gives $\alpha = 1.928780\dots$. 
In any case, the primes generated using Wright's procedure are much larger than in Mills' procedure;
with the above mentioned $\alpha$, if we denote $W_n = \lfloor \alpha_n \rfloor$, the first values are 
$W_1 = 3$, $W_2 = 13$, and $W_3 = 16381$, while $W_4$ has about $5000$ digits (OEIS sequence A016104, \cite{OEIS-Wr}).
If we try to find the smallest $\alpha$, the corresponding primes are $W_1 = 3$, $W_2 = 11$, and $W_3 = 2053$, while $W_4$ has $618$ digits, but this sequence is not included in the~OEIS.

The proofs of Mills' and Wright's results and references to some similar ones can be found in~\cite{Du}.
And some interesting variants of Mills' procedure can be seen in~\cite{El}.

There are also some prime-representing constants that allow us to find all the primes; 
here we show two examples that can be found in \cite[\S\,22.3, p.~345]{HaWr}.
The first example takes $\beta = \sum_{k=1}^{\infty} p_k/10^{2^{k}}$, 
% where $\{p_k\}$ is the sequence of all the prime numbers; 
where $p_k$ is the $k$th prime number; 
then we have
\begin{equation}
\label{eq:pnHW1}
  p_n = \lfloor {10^{2^{n}} \beta} \rfloor - 10^{2^{n-1}} \lfloor {10^{2^{n-1}} \beta} \rfloor.
\end{equation}
For the second example, let us assume that we have previously proved that, for a positive integer $r$, the primes satisfy $p_n \le r^n$ for any $n$; this is true for $r=2$ due to Bertrand's postulate, but yet more elementary for $r=4$, as we can see in \cite[Theorem~20, p.~17]{HaWr}.
Then the number $\beta = \sum_{k=1}^\infty p_k/r^{k^2}$ gives the following method to generate the primes:
\begin{equation}
\label{eq:pnHW2}
  p_n = \lfloor {r^{n^2}\beta} \rfloor - r^{2n-1} \lfloor {r^{(n-1)^2}\beta} \rfloor.
\end{equation}
In both cases, checking the expressions \eqref{eq:pnHW1} and~\eqref{eq:pnHW2} is easy; in particular, the first one is a simple manipulation of the decimal expansion of $\beta = \sum_{k=1}^\infty p_k/10^{2^{k}}$ and has the primes embedded in it.

It may seem like a dirty trick to use the primes to define $\beta$ and then manipulate $\beta$ to recover the primes but, actually, the same idea appears in Mills' and Wright's procedures. For instance, Wright starts by taking the prime $W_1 = 3$; then, using Bertrand's postulate, there exists a prime $W_2$ between $2^{W_1}$ and $2 \cdot 2^{W_1}$ (we can take $W_2=11$ or $W_2=13$), and the method continues, always taking a prime $W_{n+1}$ between $2^{W_n}$ and $2 \cdot 2^{W_n}$. Then the primes $W_n$ are used to define Wright's constant $\alpha$, and the manipulation of $\alpha$ recovers the primes. Something similar happens in the case of Mills' constant $\theta$ (with an additional difficulty due to the fact that the constant $K$ in Ingham's theorem is not made explicit). Actually, this is already remarked in \cite[\S\,22.3, p.~345]{HaWr}: ``Any one of these formulae (or any similar one) would attain a different status if the exact value [of the prime-representing constant] which occurs in it could be expressed independently of the primes. There seems no likelihood of this, but it cannot be ruled out as entirely impossible.''

In 2019, a group of students from Buenos Aires University discovered another nice prime-representing constant; see Fridman et al.~\cite{FGGGT-BA}. 
They took
\[
  \lambda = \sum_{k=1}^{\infty} \frac{p_{k}-1}{\prod_{i=1}^{k-1} p_{i}},
\]
a series whose convergence can be proved using Bertrand's postulate. It is not difficult to check that $\lambda = 2.920050977316\dots$. 
Starting with the ``Buenos Aires constant'' $\lambda$, we recursively define $\lambda_1 = \lambda$ and 
$\lambda_n = \lfloor\lambda_{n-1}\rfloor (\lambda_{n-1} - \lfloor\lambda_{n-1}\rfloor + 1)$,
and then we have $\lfloor\lambda_n\rfloor = p_n$, for $n= 1,2,\dots$. Moreover, it can be seen that $\lambda$ is an irrational number (as we have already remarked, this is not known for Mills' constant).

Just as we have wondered about the irrationality of $\lambda$, we may be interested in knowing whether it is an algebraic or a transcendental number. Let us recall that a real number is algebraic of degree $n$ if it is the root of a nonzero irreducible polynomial in $\mathbb{Z}[x]$ of degree~$n$ (in particular, the rational numbers are the algebraic numbers of degree~$1$), and a number is said to be transcendental if it is not algebraic of degree $n$ for any positive integer~$n$.

It is not known if $\lambda$ or any other of the above-mentioned prime-representing constants are algebraic or transcendental numbers. Is it possible to find a prime-representing constant that is a transcendental number? 

Only a few classes of transcendental numbers are known, in part because it can be extremely difficult to show that a given number is transcendental. In 1844, Liouville gave the first construction of a transcendental number using the idea of very fast rational approximation, a property shared by only a very thin class of transcendental numbers (known as Liouville numbers). The work of Roth in 1955 allows one to identify a wider class of transcendental numbers, but again relying on very fast rational approximation (although not necessarily as fast as in the case of Liouville numbers). We will recall both Liouville's theorem and Roth's theorem in Section~\ref{sec:theorems} of this article.

The idea for finding a transcendental prime-representing constant is to define it by means of a series, as in the construction of the numbers $\beta$ used in \eqref{eq:pnHW1} and \eqref{eq:pnHW2}, but in such a way that the series converges to a transcendental number. This happens when the series converges fast enough to allow an approximation by rationals that, according to Liouville's or Roth's theorems, generates a transcendental number. Liouville's theorem is much more elementary than Roth's, while, on the other hand, Roth's theorem is more powerful. Later we are going to give two prime-representing constants that are transcendental numbers. For the first one, its transcendence will be proved using Liouville's theorem; to prove the transcendence of the second constant, we will use Roth's theorem.

Of course, this does not imply that the constants $\beta$ used in \eqref{eq:pnHW1} and \eqref{eq:pnHW2} are algebraic numbers, only that the series that define them are not useful in proving transcendence using Liouville's or Roth's theorems. 
It is possible that these $\beta$ are transcendental numbers (actually, we think that this is likely, because, in the sense of cardinality, almost all real numbers are transcendental), although different arguments would be necessary to prove this, and finding them does not seem to be an easy task. 
But, actually, this already happens with $e = \sum_{k=0}^\infty 1/k!$, which is a transcendental number (proved by Hermite in 1873), although the series is not useful in proving transcendence.

Thus, the main theorems of this article are the following.

\begin{theorem}
\label{thm:TL}
Let $a \ge 2$ be an integer, let $\{p_k\}$ be the sequence of prime numbers, and define
\[
  \mu = \sum_{k=1}^{\infty} \frac{p_k}{a^{k+k!}}.
\]
Then $\mu$ is a transcendental number and it generates the primes by means of $p_1 = \lfloor a^2 \mu \rfloor$ and
\begin{equation}
\label{eq:TL}
  p_n = \lfloor a^{n+n!} \mu \rfloor - a^{1 + n! - (n-1)!} \lfloor a^{(n-1)+(n-1)!} \mu \rfloor
  \quad \text{for } n=2,3,\dots.
\end{equation}
\end{theorem}

\begin{theorem}
\label{thm:TR}
Let $a \ge 2$ and $m \ge 3$ be two integers, let $\{p_k\}$ the sequence of prime numbers, and define
\[
  \nu = \sum_{k=1}^{\infty} \frac{p_k}{a^{k+m^k}}.
\]
Then $\nu$ is a transcendental number and it generates the primes by means of 
$p_1 = \lfloor a^{1+m} \nu \rfloor$ and
\begin{equation}
\label{eq:TR}
  p_n = \lfloor a^{n+m^{n}} \nu \rfloor - a^{1+m^n - m^{n-1}} \lfloor a^{(n-1)+m^{n-1}} \nu \rfloor 
  \quad \text{for } n=2,3,\dots.
\end{equation}
\end{theorem}

The structure of the article is as follows. In Section~\ref{sec:lemma}, we give a lemma that will be used to prove that a constant defined as a series is a prime-representing constant. Finally, we prove Theorems~\ref{thm:TL} and~\ref{thm:TR} in Section~\ref{sec:theorems}.

%-----------
\section{A preliminary lemma}
\label{sec:lemma}
%-----------

Here, we show a general scheme to find prime-rep\-re\-sent\-ing constants defined by means of series. In the proof, we only need to use Bertrand's postulate, a consequence of which is that the $k$th prime is at most~$2^k$. It would not be very difficult to weaken the hypotheses in the lemma and give a very similar proof, but the following statement is enough for the purposes of this article. 

\begin{lemma}
Let $f : \mathbb{N} \to \mathbb{N}$ be a function that satisfies $f(k) \ge 2^{k+1} f(k-1)$ and \mbox{$f(k-1) \mid f(k)$\,} for $k \ge 2$, and define
\[
  S = \sum_{k=1}^{\infty} \frac{p_k}{f(k)}
\]
with $\{p_k\}$ the sequence of prime numbers.
Then $p_1= \lfloor f(1)S \rfloor$ and
\[
  p_n = \lfloor f(n)S \rfloor - \frac{f(n)}{f(n-1)} \lfloor f(n-1)S \rfloor
  \quad \text{for } n=2,3,\dots.
\]
\end{lemma}

\begin{proof}
Take
\[
  f(n)S = \sum_{k=1}^{n} \frac{f(n)}{f(k)} \,p_k + \sum_{k=n+1}^{\infty} \frac{f(n)}{f(k)} \,p_k.
\]
Since $f(k-1) \mid f(k)$, it follows that the first sum gives an integer;
let us analyze the second sum. 
Recall that, as a consequence of Bertrand's postulate, we have $p_k \le 2^k$.
For the first summand $k = n+1$, we have
\[
  \frac{f(n)}{f(k)} \,p_k = \frac{f(k-1)}{f(k)} \,p_k \le 2^{-k-1} \cdot 2^k = 2^{-1},
\]
and, for $k > n+1$,
\begin{align*}
  \frac{f(n)}{f(k)} \,p_k 
  &= \frac{f(n)}{f(n+1)} \frac{f(n+1)}{f(n+2)} \cdots \frac{f(k-2)}{f(k-1)} \frac{f(k-1)}{f(k)} \,p_k
  \\
  &\le 2^{-n-2} \cdot 2^{-n-3} \cdots 2^{-k-2} \cdot 2^{-k-1} \cdot 2^k
%  = 2^{-(n+2)-(n+3)-\cdots -(k+2) - 1} 
  \le 2^{-k-3},
\end{align*}
so
\[
  0 \le \sum_{k=n+1}^{\infty} \frac{f(n)}{f(k)} \,p_k 
  \le \frac{1}{2} + \sum_{k=n+2}^{\infty} \frac{1}{2^{k+3}} = \frac{1}{2} + \frac{1}{2^{n+4}} < 1.
\]
As a consequence,
\[
  \lfloor f(n)S \rfloor = \sum_{k=1}^{n} \frac{f(n)}{f(k)} \,p_k.
\]
Then, for $n>1$ we have
\begin{align*}
  &\lfloor f(n)S \rfloor - \frac{f(n)}{f(n-1)} \lfloor f(n-1)S \rfloor
  \\
  &\qquad\qquad
  = f(n) \sum_{k=1}^{n} \frac{p_k}{f(k)} - \frac{f(n)}{f(n-1)} f(n-1) \sum_{k=1}^{n-1} \frac{p_k}{f(k)}
  = p_n,
%\qedhere
\end{align*}
and the case $n=1$ is trivial.
\end{proof}

The smallest $f$ that satisfies the hypotheses of the lemma is $f(1)=1$ and $f(k) = 2^{3+4+\cdots+(k+1)}$ for $k > 1$. Other examples (that correspond to \eqref{eq:pnHW2} and \eqref{eq:pnHW1}, respectively) are $f(k) = r^{k^2}$ for any $r \ge 2$, and $f(k) = 10^{2^k}$.

Here we will use functions $f$ that grow still faster, such as
\[
  f(k) = r^k a^{k!} \qquad\text{or}\qquad f(k) = r^k a^{m^k}
\]
for $r, a, m \ge 2$, and that also satisfy the hypothesis of the lemma, but we will restrict some of the parameters to prove transcendence (in particular, we will use $a=r$ and we will take $m \ge 3$).

%-----------
\section{Proofs of the theorems}
\label{sec:theorems}
%-----------

For the sake of completeness, let us reproduce Liouville and Roth's theorems on approximation of rationals.  

Liouville's theorem was proved in 1844 and was used to prove that the number $\sum_{k=1}^{\infty} 1/10^{k!}$ (or $\sum_{k=1}^{\infty} 1/a^{k!}$ for $a\ge2$) that is not rational because its decimal expansion (respectively, its expansion in base $a$) is neither finite nor periodic, cannot be algebraic of degree $n$ for any $n\ge2$. Thus it is a transcendental number. Prior to this result, the existence of transcendental numbers was uncertain.

\begin{liouville}[1844]
Let $\xi$ be an algebraic number of degree $n \ge 2$. Then, for any fixed arbitrary constants $\varepsilon > 0$ and $K$, there are only finitely many rational approximations~$p/q$ (with $q>0$) for which
\[
  \left| \xi - \frac{p}{q} \right| < \frac{K}{q^{n+\varepsilon}}.
\]
\end{liouville}

Roth's theorem was proved in 1955 after some previous results on the order of approximation by rationals for algebraic numbers were established by Thue and Siegel; for this reason, it is sometimes known as the Thue--Siegel--Roth theorem. It can be used to prove that some series such as $\sum_{k=1}^{\infty} {1}/{a^{m^k}}$ for $a\ge2$ and $m\ge3$ converge to a transcendental number.

\begin{roth}[1955, \cite{Roth}]
Let $\xi$ be an irrational algebraic number. Then, for any arbitrary fixed $\varepsilon > 0$, there are only finitely many rational approximations~$p/q$ (with $q>0$) for which
\[
  \left| \xi - \frac{p}{q} \right| < \frac{1}{q^{2+\varepsilon}}.
\]
\end{roth}

Now we have all that we need to prove our theorems.

\begin{proof}[Proof of Theorem~\ref{thm:TL}]
It is clear that, expressed in base $a$, the digital expansion of $\mu$ is not finite or periodic, so it is not a rational number. Let us see that $\mu$ cannot be algebraic of degree $n$ for any $n \ge 2$.

For fixed $n$, let us denote the partial sums of $\mu$ by
\[
  \mu_j = \sum_{k=1}^{n+j-1} \frac{p_k}{a^k a^{k!}} = \frac{r_j}{s_j},
\]
% THE VALUE $r_j$ IS NOT IMPORTANT
where 
% $r_j = a^{(n+j-1)+(n+j-1)!} \big( p_1 a^{-1-1!} + p_2 a^{-2-2!} 
% + \cdots + p_{n+j-1} a^{-(n+j-1)-(n+j-1)!} \big)$
% and 
$s_j = a^{(n+j-1)+(n+j-1)!}$. 
By Bertand's postulate, $p_k \le 2^k \le a^k$, so we can write 
\begin{align*}
  \left| \mu - \frac{r_j}{s_j} \right|
  &= \sum_{k=n+j}^{\infty} \frac{p_k}{a^k a^{k!}} 
  \le \sum_{k=n+j}^{\infty} \frac{1}{a^{k!}} \le \sum_{k=(n+j)!}^{\infty} \frac{1}{a^{k}} \\
  &= \frac{1}{a^{(n+j)!}} \bigg( 1 + \frac{1}{a} + \frac{1}{a^2} + \cdots \bigg) 
%  = \frac{1}{a^{(n+j)!}} \frac{1}{1-1/a} 
  = \frac{1}{a^{(n+j)!}} \frac{a}{a-1}.
\end{align*}
Moreover, $s_j^{n+j} = a^{(n+j-1)(n+j)+(n+j)!}$, so
\begin{align*}
  \left| \mu - \frac{r_j}{s_j} \right| 
  &\le \frac{a^{(n+j-1)(n+j)}}{s_j^{n+j}} \frac{a}{a-1}
  = \frac{a^{(n+j-1)(n+j)+1}}{(a-1)s_j^{j-1}} \frac{1}{s_j^{n+1}} \\
  &= \frac{a^{(n+j-1)(n+j)+1}}{(a-1)a^{(j-1)(n+j-1)+(j-1)(n+j-1)!}} \frac{1}{s_j^{n+1}}
  \le \frac{K_{a,n}}{s_j^{n+1}} 
\end{align*}
for a constant $K_{a,n}$ that does not depend on~$j$.
Thus, we have infinitely many rational approximations $r_j/s_j$ and then Liouville's theorem with $\varepsilon=1$ shows that $\mu$ cannot be an algebraic number of orden~$n$.

Now, the prime representation \eqref{eq:TL} is a direct consequence of the lemma.
\end{proof}

\begin{proof}[Proof of Theorem~\ref{thm:TR}]
Expressed in base $a$, the digital expansion of $\nu$ is not finite or periodic, so it is not a rational number.

To prove transcendence, let us denote the partial sums of $\nu$ by
\[
  \nu_j = \sum_{k=1}^{j} \frac{p_k}{a^k a^{m^k}} = \frac{r_j}{s_j},
\]
% THE VALUE $r_j$ IS NOT IMPORTANT
where 
% $r_j = a^{j+m^j} \big( p_1 a^{-1-m} + p_2 a^{-2-m^2} + \cdots + p_j a^{-j-m^j} \big)$
% and 
$s_j = a^{j+m^j}$. 
By Bertand's postulate, $p_k \le 2^k \le a^k$, so we can write
%% $k=j+i$ 
\begin{align*}
  \left| \nu - \frac{r_j}{s_j} \right|
  &= \sum_{k=j+1}^{\infty} p_k a^{-k} a^{-m^k} 
  \le \sum_{k=j+1}^{\infty} a^{-m^k} 
  = \sum_{i=1}^{\infty} \big(a^{m^j}\big)^{-m^i} \\
  &= \sum_{i=1}^{\infty} \big(s_j/a^j\big)^{-m^i} 
  \le \sum_{i=1}^{\infty} \big(s_j/a^j\big)^{-mi} 
%  = \frac{(s_j/a^j)^{-m}}{1-(s_j/a^j)^{-m}}
  = \frac{1}{(s_j/a^j)^m-1}.
\end{align*}
Moreover, observe that $(s_j/a^j)^m \ge s_j^{m-1/2}$ if and only if $s_j \ge a^{2jm}$, and this inequality holds for $j \ge 2$ because $j+m^j \ge 2jm$ (recall that $m\ge3$). Then, for $j \ge 2$, 
\[
  \left| \nu - \frac{r_j}{s_j} \right| 
  \le \frac{1}{s_j^{m-1/2}-1} 
  \le \frac{1}{s_j^{m-3/4}} 
  \le \frac{1}{s_j^{2+1/4}},
\]
where in the last step we have used $m \ge 3$. 
By applying Roth's theorem with $\varepsilon=1/4$, this implies that $\nu$ is a transcendental number.

Finally, to find \eqref{eq:TR} it is enough to apply the lemma.
\end{proof}

%-----------
\begin{paragraph}{Acknowledgment.}
The research of the author is partially supported 
by grant PGC2018-096504-B-C32 of MINECO/FEDER
 (Spanish Government).
\end{paragraph}
%-----------

%-----------

%-----------

\bigskip

\noindent
Juan Luis Varona
\\*
\textit{Departamento de Matem\'aticas y Computaci\'on, Universidad de La Rioja,
\\*
26006 Logro\~no, Spain}
\\*
\textit{jvarona@unirioja.es}
%\href{https://www.unirioja.es/cu/jvarona/}{www.unirioja.es/cu/jvarona/}

%---------------
\end{document}